\documentclass[12pt, reqno]{amsart}

\author{Paolo Leonetti$^\dagger\,$}
\thanks{$\dagger\,$P.~Leonetti is supported by the Austrian Science Fund (FWF), project F5512-N26}
\address{Graz University of Technology, Institute of Analysis and Number Theory, Graz, Austria}
\email{leonetti.paolo@gmail.com}

\author{Carlo Sanna$^\ddagger$}
\thanks{$\ddagger\,$C.~Sanna is supported by a postdoctoral fellowship of INdAM and is a member of the INdAM group GNSAGA}
\address{Universit\`a degli Studi di Genova\\Department of Mathematics\\Genova, Italy}
\email{carlo.sanna.dev@gmail.com}


\title[Directions sets: A generalization of ratio sets]{Directions sets: A generalization of ratio sets}

\usepackage{amsmath}
\usepackage{amssymb}
\usepackage{amsthm}
\usepackage{bm} 
\usepackage{cite}
\usepackage{enumitem}
\usepackage[left=2.5cm, right=2.5cm, top=2.5cm, bottom=2.5cm]{geometry}
\usepackage{fancyhdr}
\usepackage{color}
\usepackage{hyperref}
\hypersetup{colorlinks=true}

\newtheorem{thm}{Theorem}[section]
\newtheorem{pro}[thm]{Proposition}
\newtheorem{cor}[thm]{Corollary}
\theoremstyle{remark}
\newtheorem{rmk}{Remark}[section]
\newtheorem{que}{Question}[section]
\newtheorem{claim}{\textsc{Claim}}

\usepackage{graphicx}

\newcommand\vect[1]{\boldsymbol{#1}}
\newcommand\accto{\mathrel{\ooalign{$\rightarrow$\cr%
  \kern.9ex\raise.7ex\hbox{\scalebox{0.500}[0.500]{$\bullet$}}\cr}}}
  
\hypersetup{
    pdftitle={{\textsc{Direction Sets: A Generalization of Ratio Sets}}},
    pdfauthor={Paolo Leonetti and Carlo Sanna},
    pdfmenubar=false,
    pdffitwindow=true,
    pdfstartview=FitH,
    colorlinks=true,
    linkcolor=blue,
    citecolor=green,
    urlcolor=cyan
}

\begin{document}

\maketitle

\begin{abstract}
For every integer $k \geq 2$ and every $A \subseteq \mathbb{N}$, we define the \emph{$k$-directions sets} of $A$ as $D^k(A) := \{\vect{a} / \|\vect{a}\| : \vect{a} \in A^k\}$ and $D^{\underline{k}}(A) := \{\vect{a} / \|\vect{a}\| : \vect{a} \in A^{\underline{k}}\}$, where $\|\cdot\|$ is the Euclidean norm and $A^{\underline{k}} := \{\vect{a} \in A^k : a_i \neq a_j \text{ for all } i \neq j\}$.
Via an appropriate homeomorphism, $D^k(A)$ is a generalization of the \emph{ratio set} $R(A) := \{a / b : a,b \in A\}$, which has been studied by many authors.
We study $D^k(A)$ and $D^{\underline{k}}(A)$ as subspaces of $S^{k-1} := \{\vect{x} \in [0,1]^k : \|\vect{x}\| = 1\}$.
In~particular, generalizing a result of Bukor and T\'oth, we provide a characterization of the sets $X \subseteq S^{k-1}$ such that there exists $A \subseteq \mathbb{N}$ satisfying $D^{\underline{k}}(A)^\prime = X$, where $Y^\prime$ denotes the set of accumulation points of $Y$. 
Moreover, we provide a simple sufficient condition for $D^k(A)$ to be dense in $S^{k-1}$.
We conclude leaving some questions for further research.
\end{abstract}

\section{Introduction}

Given $A\subseteq \mathbb{N}$, its \emph{ratio set} is defined as $R(A) := \{a / b : a,b \in A\}$.
The study of the topological properties of $R(A)$ as a subspace of $[0, +\infty]$, especially the question of when $R(A)$ is dense in $[0, +\infty]$, is a classical topic and has been considered by many researchers \cite{MR3229105, MR1390582, MR1475512, MR2843990, MR2549381, MR1197643, MR1659159, MR1904872, MR0242756, MR0248107,MR1635220,MR1328019}.
More recently, some authors have also studied $R(A)$ as a subspace of the $p$-adic numbers~$\mathbb{Q}_p$~\cite{Preprint1, MR3670202, MR3593645, MR3906498, InPress1, MR3668396}.

We consider a further variation on this theme, which stems from the following easy observation: 
We have that $[0,+\infty]$ is homeomorphic to $S^1 := \{\vect{x} \in [0,1]^2 : \|\vect{x}\| = 1\}$ via the map $x \mapsto (1, x) / \|(1,x)\|$, if $x \in {[0,+\infty)}$, and $+\infty \mapsto (0,1)$. 
This sends $R(A)$ onto $D^2(A) := \{\rho(\vect{a}): \vect{a} \in A^2\}$, where $\rho(\vect{a}):=\vect{a} / \|\vect{a}\|$ for each $\vect{a} \neq \vect{0}$. 
Hence, topological questions about $R(A)$ as a subspace of $[0, +\infty]$ are equivalent to questions about $D^2(A)$ as a subspace of $S^1$.
The novelty of this approach is that it can be generalized to higher dimensions.
For every integer $k \geq 2$, define the \emph{$k$-directions sets} of $A$ as 
\begin{equation*}
D^k(A) := \{\rho(\vect{a}): \vect{a} \in A^k\}\,\,\,\,\text{ and }\,\,\,\,D^{\underline{k}}(A) := \{\rho(\vect{a}): \vect{a} \in A^{\underline{k}}\},
\end{equation*}
where for every set $B$ we let $B^{\underline{k}} := \{\vect{b} \in B^k : b_i \neq b_j\text{ for all } i \neq j\}$ denote the set of $k$-tuples with pairwise distinct entries in $B$.
Put also $S^{k-1} := \{\vect{x} \in [0,1]^k : \|\vect{x}\| = 1\}$.
We shall study $D^k(A)$ and $D^{\underline{k}}(A)$ as subspaces of $S^{k-1}$.

Bukor and T\'{o}th~\cite{MR1390582} characterized the subsets of $[0, +\infty]$ that are equal to $R(A)^\prime$ for some $A \subseteq \mathbb{N}$, where $Y^\prime$ denotes the set of accumulation points of $Y$.
In terms of $D^2(A)$, via the homeomorphism $[0,+\infty] \to S^1$ mentioned above, their result is the following:

\begin{thm}\label{thm:ratioacc}
Let $X \subseteq S^1$.
Then there exists $A \subseteq \mathbb{N}$ such that $X = D^2(A)^\prime$ if and only if the following conditions are satisfied:
\begin{enumerate}[label={\rm \textup{(}\roman*\textup{)}}]
\item $X$ is closed;
\item $(x_1, x_2) \in X$ implies $(x_2, x_1) \in X$;
\item if $X$ is nonempty, then $(1, 0) \in X$.
\end{enumerate}
\end{thm}
Note that Theorem~\ref{thm:ratioacc} holds also if $D^{2}(A)$ is replaced by $D^{\underline{2}}(A)$.
Indeed, $D^{\underline{2}}(A)\subseteq D^{2}(A) \subseteq D^{\underline{2}}(A) \cup \{\rho(1,1)\}$ and consequently $D^{2}(A)^\prime = D^{\underline{2}}(A)^\prime$.

Our first result generalizes Theorem~\ref{thm:ratioacc}.
Before stating it, we need to introduce some notation.
Let $\vect{x} = (x_1, \dots, x_k) \in S^{k-1}$.
For every permutation $\pi$ of $\{1,\dots,k\}$, we put \mbox{$\pi(\vect{x}) := (x_{\pi(1)}, \dots, x_{\pi(k)})$.}
Also, for every $I \subseteq \{1,\dots,k\}$, we say that $I$ \emph{meets} $\vect{x}$ if there exists $j \in I$ such that $x_j \neq 0$.
In such a case, we put $\rho_{I}(\vect{x}) := \rho(\vect{y})$, where $\vect{y} = (y_1, \dots, y_k)$ is defined by $y_i := x_{i}$ if $i \in I$, and $y_i := 0$ for $i \notin I$. 
(This is well defined since $\vect{y} \neq \vect{0}$.)

Our first result is the following:

\begin{thm}\label{thm:acc}
Let $X \subseteq S^{k - 1}$ for some integer $k \geq 2$.
Then there exists $A \subseteq \mathbb{N}$ such that $X = D^{\underline{k}}(A)^\prime$ if and only if the following conditions are satisfied:
\begin{enumerate}[label={\rm \textup{(}\roman*\textup{)}}]
\item \label{item:i} $X$ is closed;
\item \label{item:ii} $\vect{x} \in X$ implies $\pi(\vect{x}) \in X$, for every permutation $\pi$ of $\{1, \dots, k\}$;
\item \label{item:iii} $\vect{x} \in X$ implies $\rho_I(\vect{x}) \in X$, for every $I \subseteq \{1, \dots, k\}$ that meets $\vect{x}$.
\end{enumerate}
\end{thm}

Note that Theorem~\ref{thm:acc} is indeed a generalization of Theorem~\ref{thm:ratioacc}, since $\rho_{I}(\vect{x}) \in \{\vect{x}, (1, 0),(0, 1)\}$ for every $I \subseteq \{1,2\}$ that meets $\vect{x} \in S^1$.
Furthermore, for $k \geq 3$, Theorem~\ref{thm:acc} is false if $D^{\underline{k}}(A)$ is replaced by $D^k(A)$ (see Remark~\ref{rmk:false} below).

Now we turn our attention to the question of when $D^k(A)$ is dense in $S^{k-1}$.
First, we have the following easy proposition.

\begin{pro}\label{prop:denseness}
Let $k \geq 2$ be an integer and fix $A \subseteq \mathbb{N}$.
We have that $D^k(A)$ is dense in $S^{k-1}$ if and only if $D^{\underline{k}}(A)$ is dense in $S^{k-1}$.
\end{pro}
\begin{proof}
On the one hand, since $D^{\underline{k}}(A) \subseteq D^k(A)$, if $D^{\underline{k}}(A)$ is dense in $S^{k-1}$ then $D^{k}(A)$ is dense in $S^{k-1}$.
On the other hand, suppose that $D^{k}(A)$ is dense in $S^{k-1}$. 
Then, for every $\vect{x} \in S^{k-1} \cap \mathbb{R}^{\underline{k}}$, there exists $\vect{a}^{(n)} \in A^k$ such that $\rho(\vect{a}^{(n)}) \to \vect{x}$.
Consequently, for all sufficiently large $n$ we have $\vect{a}^{(n)} \in A^{\underline{k}}$.
This implies that $D^{\underline{k}}(A)$ is dense in $S^{k-1} \cap \mathbb{R}^{\underline{k}}$.
Since $S^{k-1} \cap \mathbb{R}^{\underline{k}}$ is dense in $S^{k-1}$, we get that $D^{\underline{k}}(A)$ is dense in $S^{k-1}$, as desired.
\end{proof}

The next result shows that if $D^k(A)$ is dense in $S^{k-1}$, for some integer $k \geq 3$ and $A \subseteq \mathbb{N}$, then $D^{k-1}(A)$ is dense in $S^{k - 2}$, but the opposite implication is false.

\begin{thm}\label{thm:chain}
Let $k \geq 3$ be an integer.
On the one hand, if $D^k(A)$ is dense in $S^{k-1}$, for some $A \subseteq \mathbb{N}$, then $D^{k-1}(A)$ is dense in $S^{k-2}$.
On the other hand, there exists $A \subseteq \mathbb{N}$ such that $D^k(A)$ is not dense in $S^{k-1}$ but $D^{k-1}(A)$ is dense in $S^{k-2}$.
\end{thm}

We also provide a simple sufficient condition for $D^k(A)$ to be dense in $S^{k-1}$.

\begin{thm}\label{thm:limit1}
Let $A \subseteq \mathbb{N}$.
If there exists an increasing sequence $a_n \in A$ such that \mbox{$a_{n-1} / a_n \to 1$}, then $D^k(A)$ is dense in $S^{k-1}$ for every integer $k \geq 2$.
\end{thm}

The case $k = 2$ of Theorem~\ref{thm:limit1} was proved by Starni~\cite{MR1328019} (hereafter, we tacitly express all the results about $R(A)$ in terms of $D^2(A)$), who also showed that the condition is sufficient but not necessary.

Let $\mathbb{P}$ be the set of prime numbers.
It is known that $D^2(\mathbb{P})$ is dense in $S^1$~\cite{MR1197643,MR1328019} (see also~\cite{MR3115449, Preprint2, Preprint4, Preprint3} for similar results in number fields).
Let $p_n$ be the $n$th prime number.
As a consequence of the Prime Number Theorem, we have that $p_n \sim n \log n$~\cite[Theorem~8]{MR2445243}.
Hence, $p_{n-1} / p_n \to 1$ and thus Theorem~\ref{thm:limit1} yields the following:

\begin{cor}\label{cor:primes}
$D^k(\mathbb{P})$ is dense in $S^{k-1}$, for every integer $k \geq 2$.
\end{cor}

We leave the following questions to the interested readers:

\begin{que}
What is a simple characterization of the sets $X \subseteq S^{k-1}$, $k \geq 2$, such that there exists $A \subseteq \mathbb{N}$ satisfying $X = D^{k}(A)^\prime$~?
\end{que}

\begin{que}
Strauch and T\'{o}th~\cite{MR1659159} proved that if $A \subseteq \mathbb{N}$ has lower asymptotic density at least $1/2$, then $D^2(A)$ is dense in $S^1$.
Moreover, they showed that for every $\delta \in {[0, 1/2)}$ there exists some $A \subseteq \mathbb{N}$ with lower asymptotic density equal to $\delta$ and such that $D^2(A)$ is not dense in $S^1$.
How can these results be generalized to $D^k(A)$ with $k \geq 3$~?
\end{que}

\begin{que}
Bukor, \v{S}al\'{a}t, and T\'{o}th~\cite{MR1475512} proved that $\mathbb{N}$ can be partitioned into three sets $A$, $B$, $C$, such that none of $D^2(A)$, $D^2(B)$, $D^2(C)$ is dense in $S^1$.
Moreover, they showed that such a partition is impossible using only two sets.
How can these results be generalized to $D^k(A)$ with $k \geq 3$~?
\end{que}

\subsection*{Notation}
We use $\mathbb{N}$ to denote the set of positive integers.
We write vectors in bold and we use subscripts to denote their components, so that $\vect{x} = (x_1, \dots, x_k)$.
Also, we put $\|\vect{x}\| := \sqrt{x_1^2 + \cdots + x_k^2}$ for the Euclidean norm of $\vect{x}$.
If $X$ is a subset of a topological space $T$, then $X^\prime$ denotes the set of accumulation points of $X$.
Given a sequence $x^{(n)} \in T$, we write $x^{(n)} \accto x$ to mean that $x^{(n)} \to x$ as $n \to +\infty$ and $x^{(n)} \neq x$ for infinitely many $n$.

\section{Proof of Theorem~\ref{thm:acc}}

\textsc{Only If Part.} Suppose that $X = D^{\underline{k}}(A)^\prime$ for some $A \subseteq \mathbb{N}$.
We shall prove that $X$ satisfies \ref{item:i}--\ref{item:iii}.
Clearly, $X$ is closed, since it is a set of accumulation points.
Hence, \ref{item:i} holds.
Pick $\vect{x} \in X$.
Then there exists a sequence $\vect{a}^{(n)} \in A^{\underline{k}}$ such that $\rho(\vect{a}^{(n)}) \accto \vect{x}$.
In~particular, this implies that $\|\vect{a}^{(n)}\| \to +\infty$ and that $A$ is infinite. 
Let $\pi$ be a permutation of $\{1, \dots, k\}$.
Setting $\vect{b}^{(n)} := \pi(\vect{a}^{(n)})$, it follows easily that $\vect{b}^{(n)} \in A^{\underline{k}}$ and $\rho(\vect{b}^{(n)}) \accto \pi(\vect{x})$.
Consequently, $\pi(\vect{x}) \in X$ and \ref{item:ii} holds.
Finally, assume that $I \subseteq \{1, \dots, k\}$ meets $\vect{x}$.
Up to passing to a subsequence of $\vect{a}^{(n)}$, we can assume that each sequence $a_i^{(n)}$, with $i \in \{1, \dots, k\}$, is nondecreasing.
Recalling that $A$ is infinite, this implies that we can fix $k - \#I$ distinct $c_i \in A$, with $i \in \{1, \dots, k\} \setminus I$, such that $\vect{d}^{(n)} \in A^{\underline{k}}$ for every sufficiently large $n \in \mathbb{N}$, where $\vect{d}^{(n)} \in \mathbb{N}^k$ is defined by $d_i^{(n)} := a_i^{(n)}$ if $i \in I$, and $d_i^{(n)} := c_i$ if $i \notin I$.
Since $I$ meets $\vect{x}$, there exists $j \in I$ such that $x_j \neq 0$, which in turn implies that $a_j^{(n)} \to +\infty$ and consequently $\|\vect{d}^{(n)}\| \to +\infty$.
At this point, it follows easily that $\rho(\vect{d}^{(n)}) \accto \rho_I(\vect{x})$.
Hence, $\rho_I(\vect{x}) \in X$ and \ref{item:iii} holds too.

\medskip

\textsc{If Part.} Suppose that $X \subseteq S^{k-1}$ satisfies \ref{item:i}--\ref{item:iii}.
We shall prove that there exists $A \subseteq \mathbb{N}$ such that $X = D^{\underline{k}}(A)^\prime$.
Since $X$ is a closed subset of $S^{k-1}$, we have that $X$ has a countable dense subset, say $Y := \{\vect{y}^{(m)} : m \in \mathbb{N}\}$.
\begin{claim}\label{claim:sequencecm}
There exists a sequence $\vect{c}^{(m)}$ such that:
\begin{enumerate}[label={\rm (\textsc{c}\arabic{*})}]
\item \label{item:c1} $\vect{c}^{(m)}\in \mathbb{N}^{\underline{k}}$ for every $m \in \mathbb{N}$;
\item \label{item:c2}  $m \mapsto \rho(\vect{c}^{(m)})$ is an injection;
\item \label{item:c3}  $\big|\frac1{m!}c_i^{(m)} - y_i^{(m)}\big| \to 0$, for every $i \in \{1, \dots, k\}$;
\item \label{item:c4} $\big\|\rho(\vect{c}^{(m)}) - \vect{y}^{(m)}\big\| \to 0$.
\end{enumerate}
\end{claim}
\begin{proof}
For every $m \in \mathbb{N}$ and $i \in \{1, \dots, k\}$, we define $c_i^{(m)} := \lfloor m!\,y_i^{(m)} \rfloor + s_i^{(m)} + t^{(m)}$, where $\vect{s}^{(m)} \in \mathbb{N}^{k}$ and $t^{(m)} \in \mathbb{N}$ will be chosen later.
For each $m \in \mathbb{N}$, it is easy to see that we can choose $\vect{s}^{(m)} \in \{1, \dots, k\}^k$ such that $\vect{c}^{(m)} \in \mathbb{N}^{\underline{k}}$. (Note that this property does not depend on $t^{(m)}$.)
We make this choice so that \ref{item:c1} holds.
Now note that for every fixed $u,v \in \mathbb{R}^+$, with $u \neq v$, the function $\mathbb{R}^+ \to \mathbb{R} : t \mapsto \frac{u+t}{v+t}$ is injective.
Therefore, for each $m \in \mathbb{N}$ we can choose $t^{(m)} \in \{1, \dots, m\}$ such that $c_1^{(m)} / c_2^{(m)} \neq c_1^{(\ell)} / c_2^{(\ell)}$ for every positive integer $\ell < m$.
In turn, this choice implies that \ref{item:c2} holds.
At this point, both \ref{item:c3} and \ref{item:c4} follow easily.
\end{proof}

Define $A := \bigcup_{i=1}^k A_i$, where $A_i := \{c_i^{(m)} : m \in \mathbb{N}\}$ for every $i \in \{1,\dots,k\}$.
We claim that $X = D^{\underline{k}}(A)^\prime$.

\medskip

First, let us prove that $X \subseteq D^{\underline{k}}(A)^\prime$.
Pick some $\vect{x} \in X$.
Since $Y$ is a dense subset of $X$, there exists an increasing sequence of positive integers $(m_n)_{n \in \mathbb{N}}$ such that $\vect{y}^{(m_n)} \to \vect{x}$.
By the definition of $A$ and by~\ref{item:c1}, we have that $\vect{c}^{(m_n)} \in A^{\underline{k}}$.
Moreover, \ref{item:c2} and \ref{item:c4} imply that $\rho(\vect{c}^{(m_n)}) \accto \vect{x}$.
Hence, we have $\vect{x} \in D^{\underline{k}}(A)^\prime$, as desired.

\medskip

Now let us prove that $D^{\underline{k}}(A)^\prime \subseteq X$.
Pick $\vect{x} \in D^{\underline{k}}(A)^\prime$.
Then there exists a sequence $\vect{a}^{(n)} \in A^{\underline{k}}$ such that $\rho(\vect{a}^{(n)}) \accto \vect{x}$.
Up to passing to a subsequence, we can assume that there exist some $j_1, \dots, j_k \in \{1,\dots,k\}$ such that $\vect{a}^{(n)} \in A_{j_1} \times \cdots \times A_{j_k}$ for every $n \in \mathbb{N}$.
In~turn, this implies that there exists a sequence $\vect{m}^{(n)} \in \mathbb{N}^k$ such that $a_i^{(n)} = c_{j_i}^{(m_i^{(n)})}$ for every $n \in \mathbb{N}$ and $i \in \{1, \dots, k\}$.
Thanks to \ref{item:ii}, without loss of generality, we can reorder the entries of $\vect{a}^{(n)}$.
Hence, up to reordering and up to passing to a subsequence, we can assume that there exists $h \in \{1, \dots, k\}$ such that $y_{j_1}^{(m_1^{(n)})},\ldots,y_{j_h}^{(m_h^{(n)})} \neq 0$ and $y_{j_{h+1}}^{(m_{h+1}^{(n)})} = \cdots = y_{j_k}^{(m_k^{(n)})} = 0$ for every $n \in \mathbb{N}$.
Similarly, again up to reordering and up to passing to a subsequence, we can assume that there exists $\ell \in \{1, \dots, h\}$ such that $m_1^{(n)} = \cdots = m_\ell^{(n)} > m_{\ell+1}^{(n)} \geq \cdots \geq m_{h}^{(n)}$ for every $n \in \mathbb{N}$.
In~particular, since $\vect{a}^{(n)} \in A^{\underline{k}}$ for every $n \in \mathbb{N}$, we get that $j_1, \dots, j_\ell$ are pairwise distinct.
Let $\pi$ be any permutation of $\{1, \dots, k\}$ such that $\pi(i) = j_i$ for all $i \in I := \{1, \dots, \ell\}$.
Note that $I$ meets $\pi(\vect{y}^{(m_1^{(n)})})$ for every $n \in \mathbb{N}$.
Put $\vect{z}^{(n)} := \rho_I(\pi(\vect{y}^{(m_1^{(n)})}))$ for every $n \in \mathbb{N}$.
Hence, by \ref{item:ii} and \ref{item:iii} we have that $\vect{z}^{(n)} \in X$ for every $n \in \mathbb{N}$.
Thanks to \ref{item:c3}, we have that $\big|\tfrac1{m_1^{(n)}!} a_i^{(n)} - y_{j_i}^{(m_1^{(n)})}\big| \to 0$ for each $i \in I$, and $\tfrac1{m_1^{(n)}!} a_i^{(n)} \to 0$ for each $i \in \{1, \ldots, k\} \setminus I$, as $n \to +\infty$.
As a consequence, $\big\|\rho(\vect{a}^{(n)}) - \vect{z}^{(n)}\big\| \to 0$, which in turn implies that $\vect{z}^{(n)} \to \vect{x}$.
Finally, since $X$ is closed by \ref{item:i}, we obtain that $\vect{x} \in X$, as desired.

The proof is complete.

\begin{rmk}\label{rmk:false}
We note that for $k \geq 3$ the statement of Theorem~\ref{thm:acc} is false if $D^{\underline{k}}(A)$ is replaced by $D^k(A)$.
In fact, fix an integer $k \geq 3$ and let $X$ be the subset of $S^{k-1}$ containing all the permutations of $\vect{\eta}:=\rho(1,\sqrt{2},0,\ldots,0)$ and $\rho(1,0,\ldots,0)$ (and nothing else).
It follows by Theorem~\ref{thm:acc} that there exists $A\subseteq \mathbb{N}$ such that $X = D^{\underline{k}}(A)^\prime$.
For the sake of contradiction, let us suppose that there exists $B\subseteq \mathbb{N}$ such that $X = D^k(B)^\prime$. 
Since $\vect{\eta} \in X$, there exists a sequence $\vect{b}^{(n)} \in B^k$ such that $\rho(\vect{b}^{(n)}) \accto \vect{\eta}$. 
Let $\vect{c}^{(n)} \in \mathbb{N}^k$ be the sequence defined by $c_i^{(n)}=b_1^{(n)}$ if $i \neq 2$, and $c_i^{(n)} := b_2^{(n)}$ if $i = 2$.
We obtain that $\vect{c}^{(n)} \in B^k$ and $\rho(\vect{c}^{(n)}) \accto \vect{\theta}$, where $\vect{\theta} := \rho(1,\sqrt{2},1,\ldots,1)$.
(Here we have used that $\eta_1 / \eta_2$ is irrational and consequently $\rho(\vect{c}^{(n)}) \neq \vect{\theta}$.)
Therefore, $\vect{\theta} \in D^k(B)^\prime = X$, which is a contradiction.
\end{rmk}

\section{Proof of Theorem~\ref{thm:chain}}

Let $k \geq 3$ be an integer and let $A \subseteq \mathbb{N}$.
Suppose that $D^k(A)$ is dense in $S^{k-1}$.
We shall prove that $D^{k-1}(A)$ is dense in $S^{k-2}$.
For every $\vect{x} \in S^{k-2}$, let $f_k(\vect{x}) \in S^{k - 1}$ be defined by $f_k(\vect{x}) := \rho(x_1, \dots, x_{k- 1}, 0)$.
Since $D^k(A)$ is dense in $S^{k-1}$, we have that there exists a sequence $\vect{a}^{(n)} \in A^k$ such that $\rho(\vect{a}^{(n)}) \to f_k(\vect{x})$.
In turn, this implies that $\rho(\vect{b}^{(n)}) \to \vect{x}$, where $\vect{b}^{(n)} \in A^{k-1}$ is defined by $b_i^{(n)} := a_i^{(n)}$ for $i \in \{1, \dots, k - 1\}$.
Hence, $D^{k-1}(A)$ is dense in $S^{k-2}$, as desired.

Now given an integer $k \geq 3$, we shall prove that there exists $A \subseteq \mathbb{N}$ such that $D^{k-1}(A)$ is dense in $S^{k - 2}$, but $D^k(A)$ is not dense in $S^{k-1}$.
Let $X := \{\vect{x} \in S^{k - 1} : x_i = 0 \text{ for some } i\}$.
Clearly, $X$ satisfies conditions \ref{item:i}--\ref{item:iii} of Theorem~\ref{thm:acc}, and consequently there exists $A \subseteq \mathbb{N}$ such that $D^{\underline{k}}(A)^\prime = X$.
Therefore, $D^{\underline{k}}(A)$ is not dense in $S^{k-1}$ and, in light of Proposition~\ref{prop:denseness}, $D^k(A)$ is not dense in $S^{k-1}$ as well.
Finally, for every $\vect{x} \in S^{k - 2}$ we have $f_k(\vect{x}) \in X$, and the same reasonings of the previous paragraph show that $D^{k - 1}(A)$ is dense in $S^{k - 2}$.

\section{Proof of Theorem~\ref{thm:limit1}}

Suppose that there exists an increasing sequence $a_n \in A$ such that $a_{n-1} / a_n \to 1$.
Fix an integer $k \geq 2$ and pick $\vect{x} \in S^{k-1}$ with $x_1, \dots, x_k > 0$.
Clearly, for every integer $m \geq a_1\, / \min\{x_1, \dots, x_k\}$ there exist integers $m_1, \dots, m_k \geq 2$ such that $a_{m_i - 1} \leq m x_i < a_{m_i}$ for each $i \in \{1, \dots, k\}$.
Hence, for every $i \in \{1, \dots, k\}$, we have that
\begin{equation*}
x_i < \frac{a_{m_i}}{m} \leq \frac{a_{m_i}}{a_{m_i-1}} \, x_i ,
\end{equation*}
which, since $m_i \to +\infty$ as $m \to +\infty$, yields that $a_{m_i} / m \to x_i$ as $m \to +\infty$.
Putting $\vect{a}^{(m)} := (a_{m_1}, \dots, a_{m_k})$, it follows that $\rho(\vect{a}^{(m)}) \to \vect{x}$.
Therefore, $D^k(A)$ is dense in $S^{k-1}$, as claimed.

\bibliographystyle{amsplain}
\bibliography{Directions}

\end{document}